\documentclass[reqno,10pt]{amsart}
\usepackage{amsmath}
\usepackage{amsfonts}
\usepackage{amssymb}
\usepackage{amsthm}
\usepackage{newlfont}
\usepackage{graphicx}
\newtheorem{thm}{Theorem}[section]

\newtheorem{lem}[thm]{Lemma}

\theoremstyle{definition}

\theoremstyle{remark}
\newtheorem{rem}[thm]{Remark}
\numberwithin{equation}{section}

\newcommand{\ed}{\end {document}}

\title[Global $\dot H^1\cap \dot H^{-1}$  Solutions]
{Global $\dot H^1 \cap \dot H^{-1}$ solutions to a logarithmically regularized $2D$ Euler
equation}
\author[H. Dong]{Hongjie Dong}
\address[H. Dong]{Division of Applied Mathematics, Brown University,
182 George Street, Providence, RI 02912, USA}
\email{Hongjie\_Dong@brown.edu}
\author[D. Li]{Dong Li}
\address[D. Li]{Department of Mathematics, University of British Columbia, Vancouver BC Canada V6T 1Z2}%
\email{mpdongli@gmail.com}

\begin{document}
\begin{abstract}
We construct global $\dot H^1\cap \dot H^{-1}$ solutions to a logarithmically modified 2D
Euler vorticity equation. Our main tool is a new logarithm interpolation inequality which exploits
the $L^{\infty-}$-conservation of the vorticity.
\end{abstract}

\maketitle

\section{Introduction}
The usual 2D Euler equation takes the form
\begin{align} \label{V_usual}
 \begin{cases}
  \partial_t u + (u\cdot \nabla) u + \nabla p =0, \quad (t,x) \in \mathbb R \times \mathbb R^2, \\
  \nabla \cdot u =0, \\
  u \bigr|_{t=0}=u_0,
 \end{cases}
\end{align}
where $u=(u_1,u_2)$ denotes the velocity and $p$ is the pressure. Introduce the vorticity function
$\omega = -\partial_2 u_1 + \partial_1 u_2$. Then in vorticity formulation we have the equation
\begin{align} \label{E_usual}
\begin{cases}
\partial_t \omega + u \cdot \nabla \omega =0,  \quad (t,x) \in \mathbb R\times \mathbb R^2,\\
u= \nabla^{\perp} \psi=(-\partial_2 \psi, \partial_1 \psi),  \;\; \Delta \psi= \omega, \\
\omega\big|_{t=0}= \omega_0.
\end{cases}
\end{align}
Under some suitable regularity assumptions, the second equations in \eqref{E_usual} can be written as
a single equation
\begin{align}
 u= \Delta^{-1} \nabla^{\perp} \omega, \label{uw}
\end{align}
which is the usual Biot--Savart law.  We can then rewrite \eqref{E_usual} more compactly as
\begin{align*}
 \partial_t \omega + \Delta^{-1} \nabla^{\perp} \omega \cdot \nabla \omega =0.
\end{align*}

It is well-known that the system \eqref{V_usual} is globally wellposed in $H^s(\mathbb R^2)$ for any
$s>2$. See, for instance, \cite{K86, BM}. On the other hand the wellposedness in the borderline space $H^2(\mathbb R^2)$ remains unknown.
In a similar vein one can consider the wellposedness problem for
the vorticity equation \eqref{E_usual} in the borderline Sobolev spaces.
In this case since $\omega = O(\nabla u)$
it is tempting to think that local wellposedness holds in $H^s(\mathbb R^2)$ for any $s>1$.
However we should point out that this is \emph{not} the case due to some low frequency
issues introduced by the Biot--Savart relation $u=\Delta^{-1} \nabla^{\perp}\omega $. In particular under the
mere assumption $\omega \in H^s$
 the standard contraction argument no longer applies within the pure Lebesgue space framework (see Remark \ref{rem12} below for more details).
To rectify this  some amount of negative Sobolev regularity needs to be imposed on the vorticity.
For example one can prove wellposedness to \eqref{E_usual} in the space $\dot H^s \cap \dot H^{-1}$
or $\dot H^s \cap L^p$ for some $s>1$, $1<p<2$. Note that by \eqref{uw} the requirement $\omega \in \dot H^{-1} \cap \dot H^1$
is equivalent to the requirement $u \in H^2$. Thus for the vorticity equation the borderline space should be
 the space $\dot H^{-1} \cap \dot H^1$.

In this paper we consider the following generalized 2D Euler vorticity equation:
\begin{align} \label{E_new}
\begin{cases}
\partial_t \omega + u \cdot \nabla \omega = 0, \quad (t,x) \in
\mathbb R \times \mathbb R^2, \\
u = \nabla^{\perp} \psi, \; \Delta \psi = T_{\gamma} \omega, \\
\omega\big|_{t=0}=\omega_0.
\end{cases}
\end{align}
Here $T_{\gamma}=T_{\gamma}(|\nabla|)$ is a Fourier multiplier operator defined by
\begin{align*}
 \widehat{ T_{\gamma} \omega}(\xi) = \frac 1 {\log^{\gamma}(|\xi| +10)} \hat \omega(\xi)
\end{align*}
and $\gamma >0$ is a parameter. This operator introduces some additional logarithmic smoothing
of the velocity field through the second equation in \eqref{E_new}. The system \eqref{E_new} is
a model case considered in a recent paper by Chae and Wu \cite{CW_log}.
Among other results, they obtained the local wellposedness of \eqref{E_new} with initial data in the borderline Sobolev
spaces when $\gamma>1/2$. The corresponding global
wellposedness remains unknown unless some additional conditions are imposed on the initial data. Our main result is the following

\begin{thm}[Global wellposedness] \label{thm1}
Let $\gamma\ge 3/2$. Assume the initial data $\omega_0 \in \dot H^1(\mathbb R^2) \cap \dot H^{-1} (\mathbb R^2)$. Then there
exists a unique corresponding global solution $\omega$ to \eqref{E_new} in the space
$C([0,\infty), \dot H^1 \cap \dot H^{-1}) \cap C^1 ([0,\infty), L^2)$.
\end{thm}

\begin{rem} \label{rem12}
We stress that the negative regularity assumption $\omega_0 \in \dot H^{-1}(\mathbb R^2)$ is essentially
needed in Theorem \ref{thm1}. In particular it cannot be replaced by $\omega_0 \in L^2(\mathbb R^2)$.
This is due to a subtle technical issue arising from the contraction argument in the construction of
local solutions. To see it, one can consider the task of proving the uniqueness of solutions in the
space $C_t^0 H_x^1$.\footnote{The same problem will appear in the contraction argument.}
Let $\omega_1$, $\omega_2 \in C_t^0 H^1_x$ be two solutions with the same initial data $\omega_0$.
Set $\tilde \omega = \omega_1-\omega_2$. Then $\tilde \omega$ satisfies the difference equation
\begin{align}
\partial_t \tilde \omega = -\Delta^{-1} \nabla^{\perp} T_{\gamma} \tilde \omega \cdot \nabla \omega_1
- \Delta^{-1} \nabla^{\perp} T_{\gamma} \omega_2 \cdot \nabla \tilde \omega \label{ROC_1}
\end{align}
with zero initial data. To complete the proof of uniqueness one needs to compute the $L^2$-norm
of $\tilde w$ and run a Gronwall in time argument using \eqref{ROC_1}. Whilst the second term on the
RHS of \eqref{ROC_1} can be easily handled using integration by parts, there is a difficulty in controlling
the first term. Namely the advection velocity $\Delta^{-1} \nabla^{\perp} T_{\gamma} \tilde \omega$
scales like $|\nabla|^{-1} \tilde \omega$ in the low frequency regime and we cannot put it in any Lebesgue
space \emph{using only the assumption $\tilde \omega \in H^1$}. This is the main reason why we need to introduce
some amount of negative regularity on $\omega$. Of course, we can also use the space $\dot H^{-\delta}$
for some $0<\delta \le 1$ and same results can be proved. However we shall not pursue this generality here.
\end{rem}

\begin{rem}
Theorem \ref{thm1} also holds in the periodic boundary condition case. In that situation we will consider
zero mean periodic flows and the $\dot H^1$ regularity is enough to close the estimates.
It is possible to generalize our
analysis to the critical Sobolev space $\dot W^{\frac 2p, p} (\mathbb R^2) \cap \dot W^{-1,p} (\mathbb R^2)$ for
any $1<p<\infty$. However we shall not pursue this issue here.
\end{rem}

\begin{rem}
It remains a very interesting question whether the condition $\gamma\ge 3/2$ in Theorem \ref{thm1} can be relaxed. In our argument, this condition is essentially used in the proof of Lemma \ref{lem1}.
\end{rem}

\subsubsection*{Notations and Preliminaries}
\begin{itemize}
\item For any two quantities $X$ and $Y$, we denote $X \lesssim Y$ if
$X \le C Y$ for some harmless constant $C>0$. Similarly $X \gtrsim Y$ if $X
\ge CY$ for some $C>0$. We denote $X \sim Y$ if $X\lesssim Y$ and $Y
\lesssim X$. We shall write $X\lesssim_{Z_1,Z_2,\cdots, Z_k} Y$ if
$X \le CY$ and the constant $C$ depends on the quantities
$(Z_1,\cdots, Z_k)$. Similarly we define $\gtrsim_{Z_1,\cdots, Z_k}$
and $\sim_{Z_1,\cdots,Z_k}$.

\item For any $f$ on $\mathbb R^d$, we denote the Fourier transform of
$f$ has
\begin{align*}
(\mathcal F f)(\xi) = \hat f (\xi) = \int_{\mathbb R^d} f(x) e^{-i \xi \cdot x}\,dx.
\end{align*}
The inverse Fourier transform of any $g$  is given by
\begin{align*}
(\mathcal F^{-1} g)(x) = \frac 1 {(2\pi)^d} \int_{\mathbb R^d}  g(\xi) e^{i x \cdot \xi} \,d\xi.
\end{align*}


\item For any $1\le p \le \infty$ we use $\|f\|_p$, $\|f \|_{L^p(\mathbb R^d)}$, or $\|f\|_{L^p_x(\mathbb R^d)}$ to  denote the
Lebesgue norm on $\mathbb R^d$. The Sobolev space $H^1(\mathbb R^d)$ is defined
in the usual way as the completion of $C_c^{\infty}$ functions under the norm $\|f\|_{H^1} = \|f\|_2 + \| \nabla f \|_2$.
For any $s\in \mathbb R$, we define the homogeneous Sobolev norm
\begin{align*}
\| f \|_{\dot H^s} = \Bigl(\int_{\mathbb R^d} |\xi|^{2s} |\hat f(\xi)|^2 d\xi \Bigr)^{\frac 12}.
\end{align*}
For any integer $n\ge 0$ and any open set $U\subset \mathbb R^d$,
we use the notation $C^n(U)$ to denote  functions on $U$ whose $n^{th}$ derivatives
are all continuous.

\item
 We will  need to use the
Littlewood--Paley frequency projection operators. Let $\varphi(\xi)$
be a smooth bump function supported in the ball $|\xi| \leq 2$ and
equal to one on the ball $|\xi| \leq 1$. For each dyadic number $N
\in 2^{\mathbb Z}$ we define the Littlewood--Paley operators
\begin{align*}
\widehat{P_{\leq N}f}(\xi) &:=  \varphi(\xi/N)\hat f (\xi), \notag\\
\widehat{P_{> N}f}(\xi) &:=  [1-\varphi(\xi/N)]\hat f (\xi), \notag\\
\widehat{P_N f}(\xi) &:=  [\varphi(\xi/N) - \varphi (2 \xi /N)] \hat
f (\xi). 
\end{align*}
Similarly we can define $P_{<N}$, $P_{\geq N}$, and $P_{M < \cdot
\leq N} := P_{\leq N} - P_{\leq M}$, whenever $M$ and $N$ are dyadic
numbers.

\item
We recall the following Bernstein estimates:  for any $1\le p\le q\le \infty$ and dyadic $N>0$,
\begin{align}
\| P_{N} f\|_{L_x^q (\mathbb R^d)} \lesssim_d N^{d(\frac 1p -\frac 1q)} \| f\|_{L_x^p(\mathbb R^d)}. \label{0}
\end{align}
Similar inequalities also hold when $P_N$ is replaced by $P_{<N}$ or $P_{\le N}$.
\end{itemize}

\subsection*{Acknowledgements}
H. Dong was partially supported by the NSF under agreements DMS-0800129 and DMS-1056737.
D. Li was supported in part by NSF under agreement No. DMS-1128155. Any opinions, findings
and conclusions or recommendations expressed in this material are those of the authors and
do not necessarily reflect the views of the National Science Foundation.
\section{the proof}

We begin with the following simple variant of the inequality \eqref{0}.  The main example in mind
is the Fourier multiplier
\begin{align*}
 m(\xi)= \frac 1 {\log^{\gamma} (|\xi|+10)}.
\end{align*}
It is not difficult to check that $m$ satisfies the bound \eqref{0a1} below with $\tilde m(N)= \log^{-\gamma} (\frac N8+10)$.

\begin{lem} \label{lem0}
Let $m \in C^{d+1}(\mathbb R^d\setminus \{0\})$ and such that for any dyadic $N>0$, there is a constant $\tilde m(N)$ so that
\begin{align}
 \sup_{\frac N8 \le |\xi| \le 8N} | \partial_{\xi}^{\alpha} m(\xi)|
 \lesssim_d \frac {\tilde m(N)} {N^{|\alpha|} }, \qquad \forall\, |\alpha|\le d+1. \label{0a1}
\end{align}
Let $T_m$ be the associated Fourier multiplier operator defined by
\begin{align*}
 \widehat{T_m f}(\xi) = m(\xi) \hat f(\xi).
\end{align*}
Then for any dyadic $N>0$, $1\le q \le \infty$, we have
\begin{align}
 \| T_m P_N f \|_q \lesssim_d \tilde m(N) \| P_N f \|_q. \notag
\end{align}

\end{lem}

\begin{proof}[Proof of Lemma \ref{lem0}]
By inserting a fattened cut-off if necessary we only need to prove
\begin{align*}
 \| T_m P_N f \|_q \lesssim_d \tilde m(N) \| f\|_q.
\end{align*}
By a scaling argument, it suffices to show that the kernel
\begin{align*}
 K(x) = \int_{\mathbb R^d} m(N\xi) \phi(\xi)e^{i\xi \cdot x} d\xi
\end{align*}
is in $L^1(\mathbb R^d)$. Here $\phi(\xi)= \varphi(\xi)-\varphi(2\xi)$ and $\varphi$ is the same function used in the definition
of the  Littlewood--Paley projection operators. Note that $\phi$ is supported on $|\xi| \sim 1$. By \eqref{0a1}, easy to check
that
\begin{align*}
 \max_{|\alpha|\le d+1 }
 \sup_{\xi \in \mathbb R^d}\left | \partial_{\xi}^{\alpha} \Bigl( m(N\xi) \phi(\xi) \Bigr) \right| \lesssim_d
\tilde m(N).
\end{align*}
Clearly then $ x^{\alpha} K(x) \in L^{\infty}(\mathbb R^d)$ for any $|\alpha|\le d+1$. Therefore $K \in L^1$ and the desired
inequality follows from Young's inequality.
\end{proof}

\begin{lem} \label{lem0a}
For any $f \in H^1(\mathbb R^2)$, we have
\begin{align}
 \| f \|_p \le C \cdot \sqrt p \| f \|_{H^1}, \qquad \forall\, 2\le p<\infty, \label{embed0}
\end{align}
where $C>0$ is an absolute constant.
\end{lem}

\begin{proof}[Proof of Lemma \ref{lem0}]
By Bernstein, obviously
\begin{align*}
 \| P_{< 1} f \|_p \lesssim \| f\|_2.
\end{align*}
For the non-low frequency piece,  we have
\begin{align*}
 \| P_{\ge 1} f \|_p &\le \sum_{j=0}^{\infty} \| P_{2^j} f \|_p  \notag \\
& \lesssim \sum_{j=0}^{\infty} 2^{-\frac {2j}p}  2^j \| P_{2^j} f\|_2 \notag \\
& \lesssim \Bigl(\sum_{j=0}^{\infty} 2^{- \frac {4j}p} \Bigr)^{\frac 12} \| f \|_{\dot H^1} \notag \\
& \lesssim \sqrt p \| f \|_{\dot H^1}.
\end{align*}
\end{proof}

\begin{rem}
 The constant $\sqrt p$ in the inequality \eqref{embed0} is essentially sharp up to some logarithm factors
 (in terms of the dependence on $p$). To
 see this we consider  a radial function $f_p(x)=f_p(r)$ (we abuse slightly the notation here) defined by
 \begin{align} \notag
  f_p(r)= \begin{cases}
           \sqrt p, \quad r<e^{-p};\\
           \sqrt{-\log r},\quad e^{-p}\le r \le e^{-1}; \\
           \psi(r), \quad r \ge e^{-1},
          \end{cases}
 \end{align}
where $\psi$ is a smooth compactly supported function such that $\psi(e^{-1})=1$. Then easy to calculate that
$ \| f_p \|_2 \lesssim 1$  and $\| f_p \|_{\dot H^1} \lesssim \sqrt {\log p}$. On the other hand
$\| f_p \|_p \gtrsim \sqrt p$ so the sharp constant must be $\ge \sqrt{ p / \log p}$.
\end{rem}

Below is the key lemma in our proof of Theorem \ref{thm1}.
\begin{lem}  \label{lem1}
  Let $\gamma \ge \frac 32$. Then for any $f \in H^1(\mathbb R^2)$, we have
  \begin{align} \label{1}
   \Big\| \Bigl(\nabla \Delta^{-1} \nabla^{\perp} \log^{-\gamma} (|\nabla| +10)
   \Bigr)f \Big\|_{\infty} \le C_1 \cdot \log(\|f\|_{H^1} +e) \sup_{2\le p<\infty} \frac{\|f\|_p} {\sqrt p},
  \end{align}
where $C_1$ is an absolute constant.
\end{lem}
\begin{rem}
 As will become clear from the proof below, one can replace the operator $\nabla \Delta^{-1}  \nabla^{\perp}$ by any
 Riesz type operator. By Lemma \ref{lem0a} we have
 $$\sup_{2\le p<\infty} \frac{\| f \|_p}{\sqrt p} \lesssim \|f \|_{H^1}$$ so  that the RHS of \eqref{1} is well defined.
\end{rem}

\begin{proof}[Proof of Lemma \ref{lem1}]
 Denote $Tf=  \Bigl(\nabla \Delta^{-1} \nabla^{\perp} \log^{-\gamma} (|\nabla| +10)  \Bigr)f $. By Bernstein's inequality, we have
 \begin{align*}
  \| TP_{\le 2} f \|_{\infty} \lesssim \| T P_{\le 2} f\|_2  \lesssim \| f\|_2 \le \text{RHS of \eqref{1} }.
 \end{align*}
We only need to control the non-low frequency part of $f$.
Let $N$ be a dyadic number whose value will be specified later. Now split $f$ into low and high frequencies. By Lemma \ref{lem0}, we have
\begin{align*}
 \| TP_{\ge 2} f \|_{\infty} & \lesssim \sum_{j=2}^N \frac 1 {j^{\gamma}} \| P_{2^j} f \|_{\infty} +
 \sum_{j=N}^{\infty} \frac 1 {j^{\gamma}} \| P_{2^j} f \|_{\infty} \notag \\
 & \lesssim \sum_{j=2}^{N} \frac 1 {j^{\gamma}} \| f\|_{q_j} \cdot 2^{\frac {2j} {q_j} } +
 \sum_{j=N}^{\infty} \frac 1 {j^{\gamma}} \cdot 2^j \| P_{2^j} f \|_2.
\end{align*}
Choosing $q_j=j$ and using the fact that $\gamma \ge \frac 32$, we have
\begin{align*}
\| T P_{\ge 2}  f \|_{\infty} & \lesssim \sum_{j=2}^N \frac 1 {j^{\gamma-\frac 12}} \cdot \frac {\| f \|_j} {\sqrt j}
+ \Bigl( \sum_{j=N}^{\infty}  \frac 1 {j^{2\gamma}} \Bigr)^{\frac 12} \cdot \| f \|_{H^1} \notag \\
& \lesssim \log N \cdot (\sup_{2\le p<\infty} \frac {\| f \|_p} {\sqrt p}) + N^{-1} \| f \|_{H^1}.
\end{align*}
Now choose $N$ such that $N/2 < \| f \|_{H^1} +e \le N$. The desired inequality \eqref{1} follows.
\end{proof}

We are now ready to complete the

\begin{proof}[Proof of Theorem \ref{thm1}]
For the sake of completeness, we first sketch the proof of local existence and uniqueness. Start with uniqueness. Let $T_0>0$
 and let $\omega_1,\omega_2$ be two solutions to \eqref{E_new} with the same initial data $\omega_0$. The difference
 $\tilde \omega=\omega_1-\omega_2$ then satisfies the equation
 \begin{align}
\partial_t \tilde \omega = -\Delta^{-1} \nabla^{\perp} T_{\gamma} \tilde \omega \cdot \nabla \omega_1
- \Delta^{-1} \nabla^{\perp} T_{\gamma} \omega_2 \cdot \nabla \tilde \omega \notag
\end{align}
 with zero initial data. For $L^2$-norm, we compute
 \begin{align}
  \partial_t ( \| \tilde \omega \|_2^2) & \lesssim \| \Delta^{-1} \nabla^{\perp} T_{\gamma} \tilde \omega \|_{\infty}
 \| \nabla \omega_1 \|_2 \| \tilde \omega \|_2 \notag \\
  & \lesssim \| |\nabla|^{-1} \tilde \omega \|_2 \cdot \| \nabla \omega_1 \|_2 \cdot \| \tilde \omega \|_2 \notag \\
 & \qquad + \| \nabla \omega_1 \|_2 \| \tilde \omega \|_2^2. \label{s24_1}
 \end{align}
 For the $\dot H^{-1}$-norm, we have
 \begin{align}
   \partial_t ( \| \tilde \omega \|^2_{\dot H^{-1}}) & \lesssim
  \, \left| \int \Bigl( |\nabla|^{-1} \nabla \cdot ( \Delta^{-1} \nabla^{\perp} T_{\gamma} \tilde \omega \omega_1 )\Bigr)
  |\nabla|^{-1} \tilde \omega dx \right| \notag \\
  & \qquad + \left| \int \Bigl(|\nabla|^{-1} \nabla\cdot( \Delta^{-1} \nabla^{\perp} T_{\gamma} \omega_2 \tilde \omega)
 \Bigr) |\nabla|^{-1} \tilde \omega dx \right| \notag \\
 & \lesssim \, \| \Delta^{-1} \nabla^{\perp} T_{\gamma} \tilde \omega \|_{\infty} \| \omega_1 \|_2 \| \tilde \omega \|_{\dot H^{-1}} \notag\\
 & \qquad + \| \Delta^{-1} \nabla^{\perp} T_{\gamma} \omega_2 \|_{\infty} \cdot \| \tilde \omega \|_2
 \cdot \| \tilde \omega \|_{\dot H^{-1}} \notag \\
 & \lesssim \| \omega_1 \|_2 \Bigl( \| \tilde \omega \|^2_{\dot H^{-1}} + \| \tilde \omega \|_2 \| \tilde \omega \|_{\dot H^{-1}} \Bigr) \notag \\
 & \qquad + \| \omega_2 \|_{\dot H^1 \cap \dot H^{-1}} \| \tilde \omega \|_2 \cdot \|\tilde \omega \|_{\dot H^{-1} }. \label{s24_2}
 \end{align}

Adding together \eqref{s24_1} and \eqref{s24_2}, we get
\begin{align*}
 \partial_t ( \| \tilde \omega \|_2^2 + \| \tilde \omega \|^2_{\dot H^{-1}}) \lesssim_{\omega_1,\omega_2}
 \| \tilde \omega \|_2^2 + \| \tilde \omega \|^2_{\dot H^{-1}}.
\end{align*}
A simple Gronwall in time argument then yields $\tilde \omega =0$.

We turn now to the local existence in $C_t^0(\dot H^1 \cap \dot H^{-1})$.
This is fairly standard and we only sketch the details (see, for instance, \cite{BM}).
For any dyadic $N\ge 1$, consider the mollified equations\footnote{One can also use a slightly
different iteration scheme: $\partial_t \omega^{(k)} + \Delta^{-1} \nabla^{\perp} T_{\gamma} \omega^{(k-1)}
\cdot \nabla \omega^{(k)} =0$. cf. \cite{DL_fractional}.}
\begin{align*}
 \begin{cases}
  \partial_t \omega^{(N)} + P_{\le N} \Bigl(   \Delta^{-1} \nabla^{\perp} T_{\gamma} \omega^{(N)} \cdot \nabla
  P_{\le N} \omega^{(N)} \Bigr) =0, \\
  \omega^{(N)} \Bigr|_{t=0}= \omega_0,
 \end{cases}
\end{align*}
where $P_{\le N}$ is the usual Littlewood--Paley operator. By an ODE argument in Banach spaces it is easy to check
that there exists a unique solution $\omega^{(N)} \in C_t^0(\dot H^1 \cap \dot H^{-1})$. Moreover there exists
$T_0 = T_0(\| \omega_0\|_{\dot H^1 \cap \dot H^{-1}})>0$, $M_1>0$ such that
\begin{align*}
 \sup_{\substack{N>0 \\ \text{$N$ dyadic}}} \| \omega^{(N)} \|_{L_t^{\infty}{([0,T_0], \dot H^1 \cap \dot H^{-1})}} \le M_1<\infty.
\end{align*}
By using a calculation similar to \eqref{s24_1}--\eqref{s24_2}, it is not difficult to check that
$(\omega^{(N)})$ forms a Cauchy sequence in $C([0,T_0], L^2 \cap \dot  H^{-1})$ and hence admits a unique
limit point $\omega$. One can then use norm continuity along with weak continuity to show $\omega \in C([0,T_0], \dot H^1 \cap \dot H^{-1})$
 is the desired local solution. By using \eqref{E_new} it is easy to check that $\partial_t \omega \in C_t^0 L_x^2$ and
hence $\omega \in C_t^1 L_x^2$.

Finally we need to show that the local solution $\omega$ can be continued for all time. For this, it suffices to control
the $\dot H^1 \cap \dot H^{-1}$ norm  of $\omega$.

By \eqref{E_new}, we have for any $2\le p<\infty$,
\begin{align}
 \| \omega (t) \|_p \le \| \omega_0 \|_p, \qquad \forall\, t\ge 0. \label{s24_3}
\end{align}

By \eqref{s24_3}, Lemma \ref{lem1} and Lemma \ref{lem0a}, we get
\begin{align}
  \| \nabla \Delta^{-1} \nabla^{\perp} T_{\gamma} \omega(t) \|_{\infty}
  & \lesssim \log ( \| \omega \|_{H^1} +e) \cdot \sup_{2\le p<\infty} \frac{\|\omega_0\|_p}{\sqrt p} \notag \\
  & \lesssim \log ( \|\omega\|_{H^1} +e ) \cdot \| \omega_0 \|_{H^1}. \label{s24_4}
\end{align}

By \eqref{s24_3} and an argument similar to \eqref{s24_2} (one can just take $\omega_2$=0), we have
\begin{align*}
 \partial_t \Bigl( \|\omega \|^2_{\dot H^{-1}} \Bigr)
 \lesssim \| \omega_0 \|_2 \cdot \| \omega\|_{\dot H^{-1}}^2
 + \| \omega_0\|_2^2 \cdot \| \omega \|_{\dot H^{-1}}.
\end{align*}
Therefore the $\dot H^{-1}$-norm of $\omega$ is controlled for all time.

On the other hand, by \eqref{s24_4}, we have
\begin{align*}
 \partial_t \Bigl( \| \omega \|_{\dot H^1}^2 \Bigr) & \lesssim
 \| \nabla \Delta^{-1} \nabla^{\perp} T_{\gamma} \omega \|_{\infty} \cdot \| \omega \|_{\dot H^1}^2 \notag \\
 & \lesssim \, \| \omega_0 \|_{H^1} \cdot \log\left( \|\omega \|_{\dot H^1} + \| \omega_0\|_2 + e \right)
 \cdot \| \omega \|_{\dot H^1}^2.
\end{align*}
A log-Gronwall in time argument then yields that $\| \omega(t)\|_{\dot H^1}$ is bounded for all $t>0$.
{This completes the proof the theorem.}
\end{proof}

\end{document}